\documentclass{amsproc}
\usepackage{enumerate}
\usepackage{graphicx}
\usepackage{amssymb}
\usepackage{epstopdf}
\usepackage{amsmath,amsthm,amscd,amssymb}
\usepackage{latexsym}
\usepackage[colorlinks,citecolor=red,pagebackref,hypertexnames=false]{hyperref}
\usepackage{geometry}                
\numberwithin{equation}{section}

\theoremstyle{plain}
\newtheorem{theorem}{Theorem}[section]
\newtheorem{lemma}[theorem]{Lemma}

\newtheorem*{Froslemma}{Frostman Lemma}

\theoremstyle{definition}

\newtheorem{case[theorem]}{Case}

\theoremstyle{remark}

\numberwithin{equation}{section}

\def\dH{\dim_{{\mathcal H}}}
\def\R{\Bbb R}

\def\e{\epsilon}

\begin{document}

\title{Improvement on $2$-chains inside thin subsets of Euclidean spaces} 


\author{Bochen Liu}

\date{today}


\email{Bochen.Liu1989@gmail.com}
\address{Department of Mathematics, Bar-Ilan University, Ramat Gan, Israel}


\begin{abstract}
We prove that if the Hausdorff dimension of $E\subset\mathbb{R}^d$, $d\geq 2$ is greater than $\frac{d}{2}+\frac{1}{3}$, the set of gaps of $2$-chains inside $E$,
$$\Delta_2(E)=\{(|x-y|, |y-z|): x, y, z\in E \}\subset\mathbb{R}^2$$
has positive Lebesgue measure. It generalizes Wolff-Erdogan's result on distances and improves a result of Bennett, Iosevich and Taylor on finite chains.

We also consider the similarity class of $2$-chains,
$$S_2(E)=\left\{\frac{t_1}{t_2}:(t_1,t_2)\in\Delta_2(E)\right\}=\left\{\frac{|x-y|}{|y-z|}: x, y, z\in E \right\}\subset\mathbb{R},$$
and show that $|S_2(E)|>0$ whenever $\dim_{\mathcal{H}}(E)>\frac{d}{2}+\frac{1}{7}$.
\end{abstract}
\maketitle
\section{Introduction}
\subsection{Falconer distance conjecture}
Given $E\subset\R^d$, $d\geq 2$, one can define its distance set as
$$ \Delta(E)=\{|x-y|: x,y \in E \}.$$

One of the most interesting open problems in geometric measure theory is the Falconer distance conjecture, which states that $\Delta(E)$ has positive Lebesgue measure whenever $\dH(E)>\frac{d}{2}$. The best currently known results are due to Wolff  (\cite{Wol99}) in two dimensions and Erdogan (\cite{Erd05}) in higher dimensions. They proved that $|\Delta(E)|>0$ if the Hausdorff dimension of $E$ is greater than $\frac{d}{2}+\frac{1}{3}$. 

Both Wolff and Erdogan used the paradigm to attack the Falconer distance problem invented by Mattila in \cite{Mat87}. That is, to show $\Delta(E)$ has positive Lebesgue measure, it suffices to show that there exists a measure $\mu$ on $E$ such that
\begin{equation} \label{mattilaintegral} {\mathcal M}(\mu)=\int {\left( \int_{S^{d-1}} {|\widehat{\mu}(r \omega)|}^2 d\omega \right)}^2 r^{d-1} dr<\infty. \end{equation}
We call $\mathcal{M}(\mu)$ the (classical) Mattila integral. It is widely known that for any $\epsilon>0$, there exists a measure $\mu$ on $E$, called Frostmen measure, such that the energy integral
$$I_{\dH(E)-\epsilon}(\mu) = \int |\widehat{\mu}(\xi)|^2 |\xi|^{-d+\dH(E)-\epsilon}\,d\xi<\infty.  $$
What Wolff and Erdogan proved is, for this Frostman measure $\mu$, the spherical average
\begin{equation}\label{Wolff-Erdogan}
\int_{S^{d-1}} {|\widehat{\mu}(r \omega)|}^2 d\omega\lesssim r^{-\frac{d+2\dH(E)-2}{4}+\epsilon}.
\end{equation}
Then the Mattila integral ${\mathcal M}(\mu)$ is bounded above by
$$\int \left(\int_{S^{d-1}} {|\widehat{\mu}(r \omega)|}^2 d\omega\right) r^{-\frac{d+2\dH(E)-2}{4}+\epsilon} \cdot r^{d-1} dr =I_{\frac{3d-2\dH(E)+2}{4}-\epsilon}(\mu)\lesssim I_{\dH(E)-\epsilon}(\mu)<\infty  $$
if $\dH(E)>\frac{d}{2}+\frac{1}{3}$.
\subsection{Falconer-type problems and the $\frac{d+1}{2}$ barrier}
In addition to distances, one can also consider other geometric notions, such as dot products, simplices, angles, etc. (see e.g. \cite{EIT11,GI12,GGIP15,GILP15,GILP16,IME16} and references therein). An interesting fact is, among all currently known results, the dimensional exponent $\frac{d+1}{2}$ appears again and again. For example, in the paper where Falconer came up with the distance conjecture (\cite{Fal85}), he showed that $\dH(E)>\frac{d+1}{2}$ is sufficient to make sure $|\Delta(E)|>0$. For a very large class of functions $\Phi(x,y):\R^d\times\R^d\rightarrow\R$, called generalized projections, Peres and Schlag (\cite{PS00}) showed that if $\dH(E)>\frac{d+1}{2}$,
$$|\Delta^y(E)|=|\{\Phi(x,y):x\in E\}|>0$$
for a lot of points $y\in E$. In \cite{IME16}, it is shown that, the set of angles determined by $E$,
$$\{\theta(x,y,z):x,y,z\in E\}=\left\{\arccos\frac{(x-y)\cdot(z-y)}{|x-y||z-y|} :x,y,z\in E\right\} $$
has positive Lebesgue measure whenever $\dH(E)>\frac{d+1}{2}$. In \cite{BIT16}, it is shown that if $\dH(E)>\frac{d+1}{2}$, then for any $k\in\mathbb{Z}^+$, the $k$-chain set,
\begin{equation}\label{kchainset}\Delta_k(E)=\{(|x^1-x^2|, |x^2-x^3|,\dots,|x^k-x^{k+1}|): x^j\in E\}\subset\R^k\end{equation}
has positive Lebesgue measure. One can also see \cite{EIT11}, \cite{ITU16}, \cite{IL17}, where $\frac{d+1}{2}$ is obtained in different cases, with different methods. 

Due to a counterexample of Falconer (see, e.g. \cite{Fal85}), $\frac{d}{2}$ is the expected dimensional threshold for all problems above. However, except the distances, right now people get stuck at $\frac{d+1}{2}$. Therefore a reasonable short-term goal is to beat the $\frac{d+1}{2}$ barrier.

\subsection{$2$-chains inside thin subsets of Euclidean spaces}
In this paper, we beat the $\frac{d+1}{2}$ barrier on $2$-chains. Let
$$\Delta_2(E)=\{(|x-y|, |y-z|):x,y,z\in E \}.$$
The idea in Bennett-Iosevich-Taylor's $\frac{d+1}{2}$ argument on $k$-chains(\cite{BIT16}) is, if we can show that for most $x^1\in E$, the pinned distance set
$$\Delta^{x^1}(E)=\{|x^2-x^1|:x^2\in E\}$$
has positive Lebesgue measure, then induction argument works by setting $x^2$ as the new "pin". However, as we mentioned above, the best known dimensional threshold for the pinned distance problem is $\frac{d+1}{2}$ (see, e.g., \cite{PS00}, \cite{ITU16}, \cite{IL17}) and it seems very hard to improve it. Since the $2$-chain problem is weaker than the pinned distance problem, one may wonder if other ideas could help.

In this paper, with an idea of group actions that will be explained in next subsection, we shall show that $\dH(E)>\frac{d}{2}+\frac{1}{3}$ is sufficient. Notice this dimensional threshold is the same as Wolff-Erdogan's bound on distances.
\begin{theorem}
	\label{main}
	Suppose $E\subset\R^d$, $d\geq 2$ and $\dH(E)>\frac{d}{2}+\frac{1}{3}$. Then $$\Delta_2(E):=\{(|x-y|, |y-z|):x,y,z\in E \}$$ has positive $2$-dimensional Lebesgue measure.
\end{theorem}

\subsection{Similar $2$-chains}
We also consider the similarity class of $\Delta_2(E)$, that is equivalent to
$$\left\{\frac{t_1}{t_2}:(t_1,t_2)\in\Delta_2(E)\right\}=\left\{\frac{|x-y|}{|y-z|}: x, y, z\in E \right\}\subset\R.$$
\begin{theorem}
	\label{main2}
Suppose $E\subset\R^d$, $d\geq 2$ and $\dH(E)>\frac{d}{2}+\frac{1}{7}$. Then
$$S_2(E)=\left\{\frac{t_1}{t_2}:(t_1,t_2)\in\Delta_2(E)\right\}=\left\{\frac{|x-y|}{|y-z|}: x, y, z\in E \right\}\subset\R$$
has positive $1$-dimensional Lebesgue measure.
\end{theorem}

\subsection{Group actions and generalized Mattila integrals}
While the classical Mattila integral \eqref{mattilaintegral} and its connection with the distance problem can be derived directly (\cite{Mat87},\cite{Wol03}), authors in \cite{GILP15} take a geometric point of view, that has been proved so useful in the solution of the Erd\H{o}s distance conjecture (\cite{GK15}). 

Notice that $|x-y|=|x'-y'|$ if and only if there exists $\theta\in O(d)$ such that $x-\theta x'=y-\theta y'$. So we can work on the orthogonal group $O(d)$ to count the repetition of distances. With this idea, together with other brilliant ideas and techniques, Guth and Katz solved the Erd\H{o}s distance conjecture in the plane, that is, for any finite set $P\subset\R^2$,
	$$\#(\Delta(P))=\#\{|x-y|:x,y\in P\}\geq C_\epsilon \#(P)^{1-\epsilon}.$$

The key observation in \cite{GILP15} is, the Mattila integral \eqref{mattilaintegral} can be written as an integral with the Haar measure on $O(d)$ involved, i.e.
$$\int_0^\infty \left(\int_{S^{d-1}}|\hat{\mu}(r\omega)|^2\,d\omega\right)^2\,r^{d-1}\,dr= c_d\int |\hat{\mu}(\xi)|^2\left(\int_{O(d)}|\hat{\mu}(\theta\xi)|^2\,d\lambda_{O(d)}(\theta)\right)\,d\xi. $$

With this observation, authors in \cite{GILP15} developed a generalized version of Mattila integral to study the set of simplices. Recently, with a very simple argument, the author (\cite{Liu17}) gave an alternative derivation of the Mattila integral and generalizes it to the case
$$\Delta_\Phi(E_1,\dots,E_{k+1})= \{\Phi(x^1,\dots,x^{k+1}):x^j\in E_j \},$$
where $\Phi:\R^{d(k+1)}\rightarrow \R^m$ satisfies that $$\Phi(x^1,x^2,\dots,x^{k+1})=\Phi(y^1,y^2,\dots,y^{k+1})$$ if and only if $$(y^1,y^2,\dots,y^{k+1})=(gx^1,gx^2,\dots,gx^{k+1})$$ for some $g\in G$, a group admitting Haar measures.
\vskip.125in
In this paper, we consider chains inside $E$ of length $2$, that is, $\Phi(x,y,z)=(|x-y|, |y-z|)$. Unfortunately, since $x,y,z$ are not symmetric, we cannot find a group $G$ such that 
\begin{equation}\label{equal}
(|x-y|, |y-z|)= (|x'-y'|, |y'-z'|)
\end{equation}
if and only if $gx=x'$, $gy=y'$, $gz=z'$ for some $g\in G$. But the idea of group actions still helps because we can parametrize the surface
$$\{(x,y,z): (|x-y|, |y-z|)=\vec{t}\}$$ 
by
$$\{(u+x,u,u+z): u\in\R^d, (|x|, |z|)=\vec{t}\}$$ 
and
$$\{(u+\theta_1 x_{\vec{t}}, u, u+\theta_2 z_{\vec{t}}) : u\in\R^d, \theta_1, \theta_2\in O(d) \}, $$
where $(x_{\vec{t}}, z_{\vec{t}})$ is any fixed point such that $(|x_{\vec{t}}|, |z_{\vec{t}}|)=\vec{t}$.

We need more notations. Let $\phi\subset C_0^\infty$, $\int\phi=1$ and denote $\phi^\epsilon=\frac{1}{\epsilon^d}\phi(\frac{\cdot}{\epsilon})$. For any probability measure $\mu$ on $E$, denote $\mu^\epsilon = \mu*\phi^\epsilon\in C_0^\infty(\R^d)$. Define $\nu^\epsilon$ on the $\epsilon$-neighborhood of $\Delta_2(E)$ as
\begin{equation}\label{measureforchain}\int_{\R^2} F(\vec{t})\,d\nu^\epsilon(\vec{t}) = \iiint F((|x-y|, |y-z|))\,\mu^\epsilon(x)\,\mu^\epsilon(z)\,dx\,d\mu(y)\,dz.\end{equation}
\begin{theorem}
	\label{Matforchain}
With notations above.
$$\int_{\R^2} |\nu^\epsilon(\vec{t})|^2\,d\vec{t}\approx\iint \left|\int_{O(d)}\int \mu^\epsilon(u+x)\mu^\epsilon(u'+\theta x)\,dx\,d\theta \right|^2 d\mu(u)\,d\mu(u').  $$
Moreover, if there exists a measure $\mu$ on $E$ such that the right hand side is bounded above uniformly in $\epsilon$, then $\Delta_2(E)$ has positive Lebesgue measure.
\end{theorem}

For similar $2$-chains, one can define
$$\int_{\R} F(t)\,d\tilde{\nu}^\epsilon(t) = \iiint F(\frac{|x-y|}{|y-z|})\,\mu^\epsilon(x)\,\mu^\epsilon(z)\,dx\,d\mu(y)\,dz$$
and obtain the following integral.
\begin{theorem}
	\label{Matforsimilarchain}
With notations above, 
\begin{equation}\label{matdiv}\int_{\R} |\tilde{\nu}^\epsilon(t)|^2\,dt\approx\int_{\R}\iint \left|\iint \mu^\epsilon(u+x)\mu^\epsilon(u'+r\theta x)\,dx\,d\theta \right|^2 d\mu(u)\,d\mu(u')\,dr.  \end{equation}
Moreover, if the right hand side is bounded above uniformly in $\epsilon$, then $$S_2(E)=\left\{\frac{|x-y|}{|y-z|}: x, y, z\in E \right\}\subset\R$$
has positive Lebesgue measure.
\end{theorem}

\subsection{Weighted spherical averaging operators}
Let $\mu$ be a measure on $E\subset\R^d$ and $\omega_t$ be the normalized surface measure on $tS^{d-1}$, then the most natural way to define a measure $\nu$ on its distance set is, roughly speaking,
$$\nu(t)=\int\omega_t(x-y)\,d\mu(x)\,d\mu(y) = <\omega_t*\mu, 1>_\mu.  $$
Then it is natural to look at the spherical averaging operator
$$T_t f (x)= \omega_t * (f\mu)(x).$$

The key in Bennett-Iosevich-Taylor's proof on $k$-chains (\cite{BIT16}) is, if $\mu$ is a probability measure satisfying
$$\mu(B(x,r))\lesssim r^{s_\mu}, \ \forall\,x\in\R^d, \ \forall\,r>0,$$ then for each $t>0$, 
\begin{equation}\label{BITsL2}||T_t f||_{L^2(\mu)}\lesssim ||f||_{L^2(\mu)},\ \text{if}\ s_\mu>\frac{d+1}{2},\end{equation}
or equivalently,
\begin{equation}\label{BITL2}|<T_t f, g>_\mu|\lesssim ||f||_{L^2(\mu)}||g||_{L^2(\mu)},\ \text{if}\ s_\mu>\frac{d+1}{2}.\end{equation}

In fact, we can improve the dimensional exponent in \eqref{BITL2} by taking average in $t$. Define a measure $\nu$ on the distance set by
$$\nu(t)= \int\omega_t(x-y)\,f(x)d\mu(x)\,g(y)d\mu(y)= <T_t f, g>_\mu. $$
Then derivations of Mattila integrals and Lemma \ref{lemma1}, \ref{lemma3} imply
$$||\nu||_{L^2}^2\approx \int {\left( \int_{S^{d-1}} |\widehat{f\mu}(r \omega)|\,\overline{|\widehat{g\mu}(r \omega)|} d\omega \right)}^2 r^{d-1} dr\lesssim ||f||^2_{L^2(\mu)}||g||^2_{L^2(\mu)},\ \text{if}\ s_\mu>\frac{d}{2}+\frac{1}{3}.  $$
In other words,
\begin{equation}\label{averBITL2}\int|<T_t f, g>_\mu|^2\,dt\lesssim ||f||^2_{L^2(\mu)}||g||^2_{L^2(\mu)},\ \text{if}\ s_\mu>\frac{d}{2}+\frac{1}{3}.\end{equation}

In this paper, we study $2$-chains and similar $2$-chains, where measures can be defined as 
\begin{equation}\label{nu-2chain}\begin{aligned}\nu(t_1,t_2)=&\int_{\{|x-y|=t_1, |z-y|=t_2\}}\,d\mu(x)\,d\mu(y)\,d\mu(z) \\=&\int\omega_{t_1}(x-y)\,\omega_{t_2}(z-y)\,d\mu(x)\,d\mu(y)\,d\mu(z) \\= &<T_{t_1}\circ T_{t_2} 1,1 >_\mu\end{aligned}\end{equation}
and
\begin{equation}\label{nu-similar2chain}\tilde{\nu}(t)=\iint\omega_{rt}(x-y)\,\omega_{r}(z-y)\,d\mu(x)\,d\mu(y)\,d\mu(z)\,dr = \int <T_{rt}\circ T_{r} 1,1 >_\mu \,dr,\end{equation}
respectively. Then it is natural to look at $<T_{t_1}\circ T_{t_2}f, g>_\mu$ and $\int <T_{rt}\circ T_{r}f, g>_\mu dr $.

\begin{theorem}
	\label{w-ineq}
Suppose $\mu$ is a probability measure on $\R^d$ such that  
$$\mu(B(x,r))\lesssim r^{s_\mu}, \ \forall\,x\in\R^d, \ \forall\,r>0.$$ Denote $\omega_t$ as the normalized surface measure on $tS^{d-1}$ and define $T_t$ as
$$T_t f (x)= \omega_t * (f\mu)(x).$$ 
Then
\begin{equation}\label{w-ineq1}\iint |<T_{t_1}\circ T_{t_2}f, g>_\mu|^2\,dt_1\,dt_2\lesssim||f||^2_{L^2(\mu)}||g||^2_{L^2(\mu)},\ \text{if}\ s_\mu>\frac{d}{2}+\frac{1}{3}\end{equation}
and
\begin{equation}\label{w-ineq2}\int \left|\int <T_{rt}\circ T_{r}f, g>_\mu\,dr\right|^2\,dt\lesssim||f||^2_{L^2(\mu)}||g||^2_{L^2(\mu)},\ \text{if}\ s_\mu>\frac{d}{2}+\frac{1}{7}.\end{equation}
\end{theorem}

By Frostman Lemma (see Section 3), together with \eqref{nu-2chain}, \eqref{nu-similar2chain}, one can see that Theorem \ref{w-ineq} implies Theorem \ref{main}, \ref{main2}.
\vskip.125in
{\bf Notations.} Throughout this paper, $X\lesssim Y$ means $X\leq CY$ for some constant $C>0$. $X\lesssim_\epsilon Y$ means $X\leq C_\epsilon Y$ for some constant $C_\epsilon>0$, depending on $\epsilon$.
\vskip.125in
{\bf Organization.} This paper is organized as follows. In section $2$ we set up the integrals in Theorem \ref{Matforchain}, \ref{Matforsimilarchain}. In section $3$ we prove some lemmas that are useful in estimating these integrals. In section $4$ we prove Theorem \ref{w-ineq}, that implies Theorem \ref{main}, \ref{main2}.
\vskip.125in
{\bf Acknowledgements.} This work was done when the author was visiting the Hong Kong University of Science and Technology. The author would like to thank Professor Yang Wang for the financial support.

\section{Proof of Theorem \ref{Matforchain} and Theorem \ref{Matforsimilarchain}}
Denote $d\mu_f = f\,d\mu$. More generally we will consider 
\begin{equation*}\begin{aligned}\int_{\R^2} F(\vec{t})\,d\nu_{f,g}^\epsilon(\vec{t}) &= \iiint F((|x-y|, |y-z|))\,\mu_f^\epsilon(x)\,\mu_g^\epsilon(z)\,dx\,d\mu(y)\,dz,\\
\int_{\R} F(t)\,d\tilde{\nu}_{f,g}^\epsilon(t) &= \iiint F(\frac{|x-y|}{|y-z|})\,\mu_f^\epsilon(x)\,\mu_g^\epsilon(z)\,dx\,d\mu(y)\,dz,\end{aligned}\end{equation*}
and show
\begin{multline}\label{nu^2fg}
\int_{\R^2} |\nu_{f,g}^\epsilon(\vec{t})|^2\,d\vec{t}\\\approx\iint \left(\int_{O(d)}\int \mu_f^\epsilon(u+x)\mu_f^\epsilon(u'+\theta x)\,dx\,d\theta \right)\left(\int_{O(d)}\int \mu_g^\epsilon(u+x)\mu_g^\epsilon(u'+\theta x)\,dx\,d\theta \right) d\mu(u)\,d\mu(u'),\end{multline}
\begin{multline}
\int_{\R} |\tilde{\nu}_{f,g}^\epsilon(t)|^2\,dt\\\approx\iiint \left(\iint \mu_f^\epsilon(u+x)\mu_f^\epsilon(u'+r\theta x)\,dx\,d\theta \right)\left(\iint \mu_g^\epsilon(u+x)\mu_g^\epsilon(u'+r\theta x)\,dx\,d\theta \right) d\mu(u)\,d\mu(u')\,dr.\end{multline}

\subsection{Idea of the proof}
We first sketch the idea of the proof. Rigorous proof comes later in this section. 

For $2$-chians, i.e., $\nu_{f,g}(\vec{t})$, roughly speaking,
$$\nu_{f,g}(\vec{t})=\int_{\{(|x-y|,|z-y|)=\vec{t}\}}\,d\mu_f(x)\,d\mu(y)\,d\mu_g(z). $$
Since 
$$\{(u+x,u,u+z): u\in\R^d, (|x|, |z|)=\vec{t}\}=\{(u+\theta_1 x_{\vec{t}}, u, u+\theta_2 z_{\vec{t}}) : u\in\R^d, \theta_1, \theta_2\in O(d) \}, $$
where $(x_{\vec{t}}, z_{\vec{t}})$ is any fixed point such that $(|x_{\vec{t}}|, |z_{\vec{t}}|)=\vec{t}$, we have two expressions of $\nu_{f,g}$ on $\Delta_2(E)$,
$$\nu_{f,g}(\vec{t}) \approx \int_E\int_{\{(|x|,|z|)=\vec{t}\}} \mu_f(u+x) \mu_g(u+z)\,d\mathcal{H}^{2d-2}(x,z)\,d\mu(u)$$
and
$$\nu_{f,g}(\vec{t}) \approx \int_E\int_{O(d)}\int_{O(d)} \mu_f(u'+\theta_1 x_{\vec{t}}) \mu_g(u'+\theta_2 z_{\vec{t}})\,d\theta_1\,d\theta_2\,d\mu(u').$$
Multiplying these two expressions and integrate it in $\vec{t}$, it follows that the square of the $L^2$-norm of $\nu_{f,g}$ approximately equals
$$ \iiint\!\!\!\!\int\limits_{\{(|x|,|z|)=\vec{t}\}}\!\!\!\!\! \mu_f(u+x) \mu_g(u+z)\left(\iint\mu_f(u'+\theta_1 x_{\vec{t}})\, \mu_g(u'+\theta_2 z_{\vec{t}})\,d\theta_1\,d\theta_2\right)\!d\mathcal{H}^{2d-2}(x,z)\,d\vec{t}\,d\mu(u)\,d\mu(u'). $$
By the invariance of $d\theta$, on the surface $\{(|x|,|z|)=\vec{t}\}$, we may replace $x_{\vec{t}}$ by $x$, $z_{\vec{t}}$ by $z$. Also $$d\mathcal{H}^{2d-2}|_{\{(|x|,|z|)=\vec{t}\}}(x,z)\, d\vec{t}\approx dx\,dz.$$
Therefore the integral above approximately equals
\begin{equation*}
\begin{aligned}&\iiiint\mu_f(u+x) \mu_g(u+z)\left(\iint\mu_f(u'+\theta_1 x) \mu_g(u'+\theta_2z)\,d\theta_1\,d\theta_2\right)dx\,dz\,d\mu(u)\,d\mu(u')\\=&
\iint \left(\int_{O(d)}\int \mu_f(u+x)\mu_f(u'+\theta x)\,dx\,d\theta \right)\left(\int_{O(d)}\int \mu_g(u+x)\mu_g(u'+\theta x)\,dx\,d\theta \right) d\mu(u)\,d\mu(u'),
\end{aligned}
\end{equation*}
as desired.
\vskip.125in
It is quite similar in the case of similar $2$-chains. Notice
$$\{(u+x,u,u+z): u\in\R^d, |x|=t|z|\}=\{(u+r\theta_1 x_{t}, u, u+r\theta_2 z_{t}) : u\in\R^d, r\in\R, \theta_1, \theta_2\in O(d) \}, $$
where $(x_{t}, z_{t})$ is any non-zero fixed point such that $|x_{t}|=t|z_{t}|$. Then we have two ways to express a measure $\tilde{\nu}_{f,g}$ on $S_2(E)$,
$$\tilde{\nu}_{f,g}(t) \approx \int_E\int_{\{|x|=t|z|)\}} \mu_f(u+x) \mu_g(u+z)\,d\mathcal{H}^{2d-1}(x,z)\,d\mu(u),$$
$$\tilde{\nu}_{f,g}(t) \approx \int_E\int_{O(d)}\int_{O(d)} \mu_f(u'+r\theta_1 x_{t}) \mu_g(u'+r\theta_2 z_{t})\,d\theta_1\,d\theta_2\,dr\,d\mu(u'),$$
and all steps above still work.

\subsection{Rigorous proof}For a rigorous proof, we need the coarea formula. For smooth cases the coarea formula follows from a simple change of variables. More general forms of the formula for Lipschitz functions were first established by Federer in 1959 and later generalized by different authors. For references, one can see \cite{Fed69}. We will use the following version in this paper.

\begin{theorem}
	[Coarea formula, 1960s]
Let $\Phi$ be a Lipschitz function defined in a domain $\Omega\subset\R^{d(k+1)}$, taking on values in $\R^m$ where $m<d(k+1)$. Then for any $f\in L^1(\R^{d(k+1)})$,
$$\int_{\Omega} F(x)|J_m\Phi(x)|\,dx=\int_{\R^{m}}\left(\int_{\Phi^{-1}(\vec{t})} F(x)\,d\mathcal{H}^{d(k+1)-m}(x)\right)\,d\vec{t},  $$
where $J_m\Phi$ is the $m$-dimensional Jacobian of $\Phi$ and $\mathcal{H}^{d(k+1)-m}$ is the $(d(k+1)-m)$-dimensional Hausdorff measure.
\end{theorem}

We only prove Theorem \ref{Matforchain}. Then Theorem \ref{Matforsimilarchain} follows in a very similar way. Denote $\Phi_y(x,z)=(|x-y|,|z-y|)$. Fix $y$ and apply the the coarea formula on $dx\,dz$, \eqref{measureforchain} can be written as
\begin{equation*} 
\int_{\R^2} F(\vec{t})\, d\nu_{f,g}^\epsilon(\vec{t})=\int F(\vec{t}) \left(\int \int_{\Phi_y^{-1}(\vec{t})} \mu_f^\epsilon(x)\,\mu_g^\epsilon(z)\,\frac{1}{|J_2\Phi_y(x,z)|}\,d\mathcal{H}^{2d-2}(x,z)\,d\mu(y)\right)\,d\vec{t}.
\end{equation*}
It follows that 
\begin{equation}\label{nu1}
\nu_{f,g}^\epsilon(\vec{t})=\int_E\int_{\Phi_y^{-1}(\vec{t})} \,\mu_f^\epsilon(x)\,\mu_g^\epsilon(z)\,\frac{1}{|J_2\Phi_y|}\,d\mathcal{H}^{2d-2}(x,z)\,d\mu(y).
\end{equation}
On the other hand, the probability Haar measure $d\theta$ on $O(d)$ induces a measure $d\sigma^y_{\vec{t}}\,d\mu$ on $\Phi_{y}^{-1}(\vec{t})\cap\{y\in E\}$ by
$$\int_E\int_{\Phi_{y}^{-1}(\vec{t})} F(x,z)\, d\sigma^y_{\vec{t}}(x,z)\,d\mu(y) = \int_E\int_{O(d)}\int_{O(d)} F(\theta_1 (x^{y}_{\vec{t}}-y)+y, \theta_2 (z^{y}_{\vec{t}}-y)+y)\, d\theta_1\,d\theta_2\,d\mu(y), $$
where $(x^y_{\vec{t}}, z^y_{\vec{t}})$ is any fixed point such that $(|x^{y}_{\vec{t}}-y|, |z^{y}_{\vec{t}}-y|)=\vec{t}$. By the invariance of the Haar measure, with $y$ fixed, $\sigma^y_{\vec{t}}$ does not depedent on the choice of $(x^y_{\vec{t}}, z^y_{\vec{t}})$, 
and it must be absolutely continuous with respect to 
$\mathcal{H}^{2d-2}|_{\Phi_y^{-1}(\vec{t})}$. In fact, by the invariance, there exists a positive  function $\psi$ on $\Phi_y^{-1}(\vec{t})$ such that
$$\sigma^y_{\vec{t}}=\psi\, \mathcal{H}^{2d-2}|_{\Phi_y^{-1}(\vec{t})}. $$

On any compact set, $\psi\approx 1$, so another expression of $\nu_{f,g}^\e$ follows,
\begin{equation}\label{nu2}
\begin{aligned}
\nu_{f,g}^\epsilon(\vec{t})=&\int_E\int_{\Phi_y^{-1}(\vec{t})} \,\mu_f^\epsilon(x)\,\mu_g^\epsilon(z)\,\frac{1}{|J_2\Phi_y|}\,d\mathcal{H}^{2d-2}(x,z)\,d\mu(y)\\\approx& \int_E\int_{\Phi_y^{-1}(\vec{t})} \,\mu_f^\epsilon(x)\,\mu_g^\epsilon(z)\,d\sigma^y_{\vec{t}}(x,z)\,d\mu(y)\\=&\int_E\int_{O(d)}\int_{O(d)} \mu_f^\epsilon(\theta_1 (x^{y}_{\vec{t}}-y)+y)\,\mu_g^\epsilon (\theta_2 (z^{y}_{\vec{t}}-y)+y)\, d\theta_1\,d\theta_2\,d\mu(y)
\end{aligned}
\end{equation}
\vskip.125in

For convenience, we change $y$ in \eqref{nu1} by $u$ and $y$ in \eqref{nu2} by $u'$, then $\int |\nu^\e|^2$ is approximately
\begin{multline*}
		\iiint\int_{\Phi_u^{-1}(\vec{t})} \,\mu_f^\epsilon(x)\,\mu_g^\epsilon(z)\left(\int_{O(d)}\int_{O(d)} \mu_f^\epsilon(\theta_1 (x^{u'}_{\vec{t}}-u')+u')\,\mu_g^\epsilon (\theta_2 (z^{u'}_{\vec{t}}-u')+u')\, d\theta_1\,d\theta_2\right)\\\frac{1}{|J_2\Phi_u|}\,d\mathcal{H}^{2d-2}(x,z)\,d\vec{t}\,d\mu(u)\,d\mu(u').
\end{multline*}

For any pair $(x,z)\in \Phi_u^{-1}(\vec{t})$, by definition $(|x-u|, |z-u|)=\vec{t}$ and therefore $(|(x-u+u')-u'|, |(z-u+u')-u'|)=\vec{t}$. By the invariance of the Haar measure we may replace $x^{u'}_{\vec{t}}$ by $x-u+u'$ and $z^{u'}_{\vec{t}}$ by $z-u+u'$,
\begin{multline*}
\int_E\int_E\int_{\R^2}\int_{\Phi_u^{-1}(\vec{t})}\mu_f^\epsilon(x)\mu_g^\epsilon(z)\left(\int_{O(d)}\int_{O(d)}\mu_f^\epsilon(\theta_1 (x-u)+u')\mu_g^\epsilon (\theta_2 (z-u)+u')d\theta_1 d\theta_2\right)\\\frac{1}{|J_2\Phi_u|}d\mathcal{H}^{2d-2}(x,z)\,d\vec{t}\,d\mu(u)d\mu(u').
\end{multline*}
By the coarea formula, $\frac{1}{|J_2\Phi_u|}d\mathcal{H}^{2d-2}|_{\Phi_u^{-1}(\vec{t})}(x,z)\,d\vec{t} = dx\,dz$, then it equals
\begin{equation*}
\begin{aligned}
&\iiiint\mu_f^\epsilon(x)\mu_g^\epsilon(z)\!\left(\int_{O(d)}\int_{O(d)} \mu_f^\epsilon(\theta_1 (x-u)+u')\mu_g^\epsilon (\theta_2 (z-u)+u')d\theta_1\,d\theta_2\right)\!dx\,dz\,d\mu(u)\,d\mu(u')\\=&\iint\iint\mu_f^\epsilon(x+u)\mu_g^\epsilon(z+u)\left(\int_{O(d)} \mu_f^\epsilon(\theta x+u')\mu_g^\epsilon (\theta z+u')d\theta\right)dx\,dz\,d\mu(u)\,d\mu(u')\\=&\iint \left(\int_{O(d)}\int \mu_f^\epsilon(x+u)\mu_f^\epsilon(\theta x+u')\,dx\,d\theta \right)\left(\int_{O(d)}\int \mu_g^\epsilon(x+u)\mu_g^\epsilon(\theta x+u')\,dx\,d\theta \right) d\mu(u)\,d\mu(u'),
\end{aligned}
\end{equation*}
as desired.

\section{Restriction-type lemmas}
We need natural measures on $E$ illustrating its Hausdorff dimension.

\begin{Froslemma}
	[see, e.g. \cite{Mat95}]
	Suppose $E\subset\R^d$ and denote $\mathcal{H}^s$ as the $s$-dimensional Hausdorff measure. Then $\mathcal{H}^s(E)>0$ if and only if there exists a probability measure $\mu$ on $E$ such that 
$$\mu(B(x,r))\lesssim r^s $$
	for any $x\in\R^d$, $r>0$.
\end{Froslemma}

Since by definition $\dH(E)=\sup\{s:\mathcal{H}^s(E)>0\}$, Frostman Lemma implies that for any $s_\mu<\dH(E)$ there exists a probability measure $\mu_E$ on $E$ such that
\begin{equation}\label{Frostmanmeasure}\mu_E(B(x,r))\lesssim r^{s_\mu},\ \forall\ x\in\R^d,\ r>0.  
\end{equation}

We need the following restriction-type lemmas on measures satisfying \eqref{Frostmanmeasure}.
\begin{lemma}\label{lemma1}
Suppose $\mu$ satisfies \eqref{Frostmanmeasure}. Then
	$$\int_{|\xi|\leq R}|\widehat{f\, d\mu}(\xi)|^2\,d\xi\lesssim R^{d-s_\mu}||f||_{L^2(\mu)}^2. $$
\end{lemma}
\begin{proof}
Take $\psi\subset C_0^\infty(\R^d)$ whose Fourier transform is positive in the unit ball. Then
\begin{equation*}
\begin{aligned}
	\int_{|\xi|\leq R}|\widehat{f\,d\mu}(\xi)|^2\,d\xi&\lesssim \int|\widehat{f\,d\mu}(\xi)|^2\,\widehat{\psi}(\frac{\xi}{R})\,d\xi\\& = R^d\iint \psi(R(x-y))\,f(x)f(y)d\mu(x)\,d\mu(y).
\end{aligned}	
\end{equation*}
Since $\psi$ has bounded support and $\mu$ satisfies \eqref{Frostmanmeasure},
$$\int|\psi(R(x-y))|\,d\mu(x)\lesssim R^{-s_\mu},\ \int|\psi(R(x-y))|\,d\mu(y)\lesssim R^{-s_\mu}.$$
Then the lemma follows by Shur's test.
\end{proof}

\begin{lemma}
	\label{lemma2}
Suppose $f$ is supported on $\{\xi\in\R^d: |\xi|\leq R\}$ and $\mu$ satisfies \eqref{Frostmanmeasure}, then
$$\int|\hat{f}|^2\, d\mu\lesssim R^{d-s_\mu}||f||_{L^2(\R^d)}^2. $$
\end{lemma}
\begin{proof}
	For any $h \in L^2(\mu)$, $||h||_{L^2(\mu)}=1$,
$$\left(\int \hat{f}\, h\,d\mu\right)^2=\left(\int_{|\xi|\leq R} f\,\widehat{h\,d\mu}\right)^2\leq ||f||_{L^2(\R^d)}^2 \int_{|\xi|\leq R} |\widehat{h\,d\mu}|^2\lesssim R^{d-s_\mu}||f||_{L^2(\R^d)}^2,  $$
where the last inequality follows from Lemma \ref{lemma1}.
\end{proof}

Denote $d\omega_R$ as the normalized surface measure on $RS^{d-1}$, the sphere of radius $R$ centered at the origin.

\begin{lemma}\label{lemma3}
Suppose $E\subset\R^d$ and $\mu$ satisfies \eqref{Frostmanmeasure}, then
	$$\int_{RS^{d-1}}|\widehat{f\, d\mu}|^2\,d\omega_R\lesssim_\epsilon R^{-\frac{d+2s_\mu-2}{4}+\epsilon}||f||_{L^2(\mu)}^2. $$
\end{lemma}
\begin{proof}
Denote $A_R$ as the $1$-neighborhood of $RS^{d-1}$. The strategy Wolff and Erdogan used to prove \eqref{Wolff-Erdogan} is the following. First by uncertainty principle,
$$\int_{RS^{d-1}}|\widehat{\mu}|^2\,d\omega_R\approx R^{-(d-1)}\int_{A_R}|\widehat{\mu}(\xi)|^2\,d\xi.$$ 
Then
$$\int_{A_R}|\widehat{\mu}|^2=\sup_{\substack{supp\,h\subset A_R\\||h||_2=1}} \left(\int h\,\widehat{\mu}\right)^2 = \sup_{\substack{supp\,h\subset A_R \\ ||h||_2=1}} \left(\int \hat{h}\,d\mu\right)^2\leq \sup_{\substack{supp\,h\subset A_R\\||h||_2=1}} \int |\hat{h}|^2\,d\mu. $$
What Wolff and Erdogan proved is, for any $h$ supported on $A_R$,
\begin{equation}\label{Wolff-Erdogan1}\int |\hat{h}|^2\,d\mu\lesssim_\epsilon R^{d-1-\frac{d+2 s_\mu-2}{4}+\epsilon}||h||_{L^2}^2.
\end{equation}

In our case, the uncertainty principle still works, i.e.,
$$\int_{RS^{d-1}}|\widehat{f\,d\mu}|^2\,d\omega_R\approx R^{-(d-1)}\int_{A_R}|\widehat{f\,d\mu}(\xi)|^2\,d\xi.$$
Then
$$\int_{A_R}|\widehat{f\,d\mu}|^2=\sup_{\substack{supp\,h\subset A_R\\||h||_2=1}} \left(\int h\,\widehat{f\,d\mu}\right)^2 = \sup_{\substack{supp\,h\subset A_R \\ ||h||_2=1}} \left(\int \hat{h}f\,d\mu\right)^2\leq ||f||_{L^2(\mu)}^2 \sup_{\substack{supp\,h\subset A_R\\||h||_2=1}} \int |\hat{h}|^2\,d\mu$$
and the lemma follows from Wolff-Erdogan's estimate \eqref{Wolff-Erdogan1}.
\end{proof}

\begin{lemma}
	\label{lemma4}
	Suppose $E\subset\R^d$ and $\mu$ satisfies \eqref{Frostmanmeasure}, then
$$\int |\widehat{\widehat{f\,d\mu}\,d\omega_R}|^2\,d\mu\lesssim_\epsilon R^{-\frac{d+2 s_\mu-2}{4}+\epsilon} \int_{RS^{d-1}}|\widehat{f\,d\mu}|^2\,d\omega_R. $$
\end{lemma}
\begin{proof}
	For any $h\in L^2(\mu)$, $||h||_{L^2(\mu)}=1$,
$$\left(\int \widehat{\widehat{f\,d\mu}\,d\omega_R}\,h\,d\mu\right)^2 = \left(\int_{RS^{d-1}} \widehat{h\,d\mu}\,\widehat{f\,d\mu}\,d\omega_R\right)^2\leq \left(\int_{RS^{d-1}}|\widehat{h\, d\mu}|^2\,d\omega_R\right)\left(\int_{RS^{d-1}}|\widehat{f\,d\mu}|^2\,d\omega_R\right).  $$
Since $||h||_{L^2(\mu)}=1$, by Lemma \ref{lemma3}, it is bounded above by
$$R^{-\frac{d+2 s_\mu-2}{4}+\epsilon} \int_{RS^{d-1}}|\widehat{f\,d\mu}|^2\,d\omega_R,  $$
as desired.
\end{proof}

\section{Proof of Theorem \ref{w-ineq}}
\subsection{Proof of \eqref{w-ineq1}}
Similar to \eqref{nu-2chain}, one can see
$$\iint |\nu_{f,g}(t_1, t_2)|^2\,dt_1\,dt_2= \iint |<T_{t_1}\circ T_{t_2}f, g>_\mu|^2\,dt_1\,dt_2. $$
Thus by \eqref{nu^2fg} and Cauchy-Schwartz, it suffices to show 
$$\mathcal{M}^\epsilon=\iint \left|\int_{O(d)}\int \mu_f^\epsilon(u+x)\mu_f^\epsilon(u'+\theta x)\,dx\,d\theta \right|^2 d\mu(u)\,d\mu(u') \lesssim ||f||^2_{L^2(\mu)},\ \text{if}\ s_\mu>\frac{d}{2}+\frac{1}{3}. $$

By Plancherel in $x$, it equals
$$\iint\left|\int_{O(d)}\int\widehat{\mu_f^\epsilon}(\xi)\,e^{-2\pi i \xi\cdot u'}\, \widehat{\mu_f^\epsilon}(-\theta^t\xi) \,e^{2\pi i \xi\cdot (\theta u)}\,d\xi\,d\theta\right|^2\,d\mu(u)\,d\mu(u').$$

Since $\widehat{\mu_f^\epsilon}(\xi) = \widehat{f\,d\mu}(\xi)\widehat{\phi}(\epsilon \xi)$, $\phi\in C_0^\infty$, it suffices to show that
\begin{equation}\label{LPpiece}\begin{aligned}\mathcal{M}_j = &\iint\left|\int_{O(d)}\int_{|\xi|\approx 2^j}\widehat{f\,d\mu}(\xi)\,e^{-2\pi i \xi\cdot u'}\, \widehat{f\,d\mu}(-\theta^t\xi) \,e^{2\pi i \xi\cdot (\theta u)}\,d\xi\,d\theta\right|^2\,d\mu(u)\,d\mu(u')\\\lesssim & \ 2^{-j\gamma}||f||^2_{L^2(\mu)} \end{aligned}\end{equation}
for some $\gamma>0$. Then $\sqrt{\mathcal{M}^\epsilon}\leq\sum\limits_{j\leq -\log \epsilon}\sqrt{\mathcal{M}_j}<\sum\sqrt{\mathcal{M}_j} \lesssim||f||^2_{L^2(\mu)}$.

Denote
$$F^j_{u}(\xi) = \chi_{\{|\xi|\approx 2^j\}}\, \widehat{f\,d\mu}(\xi)\int_{O(d)}\widehat{f\,d\mu}(-\theta^t\xi)\,e^{-2\pi i \xi\cdot (\theta u)}\,d\theta= \chi_{\{|\xi|\approx 2^j\}}\, \widehat{f\,d\mu}(\xi)\,\widehat{\widehat{f\,d\mu}\,d\omega_{|\xi|}}(u).$$

Then 
$$\mathcal{M}_j= \iint|\widehat{F^j_{u}}(u')|^2\,d\mu(u')\,d\mu(u)  $$
and by Lemma \ref{lemma2} it is less than or equal to
\begin{equation*}
\begin{aligned}
&2^{j(d-s_\mu)}\iint|F^j_{u}(\xi)|^2\,d\xi\,d\mu(u)\\=&2^{j(d-s_\mu)} \int\int_{|\xi|\approx 2^j} |\widehat{f\,d\mu}(\xi)|^2 \,|\widehat{\widehat{f\,d\mu}\,d\omega_{|\xi|}}(u)|^2\,d\xi\,d\mu(u)\\=&2^{j(d-s_\mu)} \int_{|\xi|\approx 2^j} |\widehat{f\,d\mu}(\xi)|^2 \,\left(\int|\widehat{\widehat{f\,d\mu}\,d\omega_{|\xi|}}(u)|^2\,d\mu(u)\right)d\xi.
\end{aligned}
\end{equation*}

Fix $\xi$, by Lemma \ref{lemma4},
\begin{equation}
	\label{common}
	\begin{aligned}
	\mathcal{M}_j&\lesssim_\epsilon 2^{j(d-s_\mu)} \int_{|\xi|\approx 2^j} |\widehat{f\,d\mu}(\xi)|^2 |\xi|^{-\frac{d+2s_\mu-2}{4}+\epsilon} \left(\int_{|\xi|S^{d-1}}|\widehat{f\,d\mu}|^2\,d\omega_{|\xi|} \right)\,d\xi\\&\approx\  2^{j(d-s_\mu-\frac{d+2s_\mu-2}{4}+\epsilon)} \int_{|\xi|\approx 2^j} |\widehat{f\,d\mu}(\xi)|^2 \left(\int_{|\xi|S^{d-1}}|\widehat{f\,d\mu}|^2\,d\omega_{|\xi|} \right)\,d\xi.
	\end{aligned}
\end{equation}

Apply Lemma \ref{lemma3}, \ref{lemma1}, it follows that
\begin{equation*}
\begin{aligned}\mathcal{M}_j &\lesssim_\epsilon 2^{j(d-s_\mu-\frac{d+2s_\mu-2}{4}+\epsilon)}\cdot 2^{j(-\frac{d+2s_\mu-2}{4}+\epsilon)}\cdot 2^{j(d-s_\mu)}||f||^2_{L^2(\mu)}\\&=\  2^{j(\frac{3d-6s_\mu-2}{2}+\epsilon)}||f||^2_{L^2(\mu)},
\end{aligned}
\end{equation*}
which is summable in $j$ when $s_\mu>\frac{d}{2}+\frac{1}{3}$, as desired.
\subsection{Proof of \eqref{w-ineq2}}
The proof of \eqref{w-ineq2} is quite similar. It suffices to show 
$$\iiint \left|\int_{O(d)}\int \mu_f^\epsilon(u+x)\mu_f^\epsilon(u'+r\theta x)\,dx\,d\theta \right|^2 d\mu(u)\,d\mu(u')\,dr \lesssim ||f||^2_{L^2(\mu)},\ \text{if}\ s_\mu>\frac{d}{2}+\frac{1}{7}. $$
By Plancherel in $x$, one gets
$$\int_{r\approx 1}\iint\left|\int_{O(d)}\int\widehat{\mu_f^\epsilon}(\xi)\,e^{-2\pi i \xi\cdot u'}\, \widehat{\mu_f^\epsilon}(-r\theta^t\xi) \,e^{2\pi i \xi\cdot (r\theta u)}\,d\xi\,d\theta\right|^2\,d\mu(u)\,d\mu(u')\,dr$$and it suffices to show
\begin{equation*}\begin{aligned}\mathcal{M}_j = &\int_{r\approx 1}\iint\left|\int_{O(d)}\int_{|\xi|\approx 2^j}\widehat{f\,d\mu}(\xi)\,e^{-2\pi i \xi\cdot u'}\, \widehat{f\,d\mu}(-r\theta^t\xi) \,e^{2\pi i \xi\cdot (r\theta u)}\,d\xi\,d\theta\right|^2\,d\mu(u)\,d\mu(u')\,dr\\\lesssim &\ 2^{-j\gamma}||f||^2_{L^2(\mu)}\end{aligned}\end{equation*}
for some $\gamma>0$. Compared with the $\mathcal{M}_j$ above in \eqref{LPpiece}, the only difference is that we need to integrate $dr$. So every step in the last subsection before \eqref{common} still works and \eqref{common} becomes
\begin{equation*}
\begin{aligned}
&2^{j(d-s_\mu-\frac{d+2s_\mu-2}{4}+\epsilon)} \int_{|\xi|\approx 2^j} |\widehat{f\,d\mu}(\xi)|^2 \left(\int_{r\approx 1}\int_{RS^{d-1}}|\widehat{f\,d\mu}(r|\xi|\omega)|^2\,d\omega\,dr \right)\,d\xi \\ \approx &\,2^{j(d-s_\mu-\frac{d+2s_\mu-2}{4}+\epsilon)} \int_{|\xi|\approx 2^j} |\widehat{f\,d\mu}(\xi)|^2 \,|\xi|^{-d}\left(\int_{|\eta|\approx 2^j}|\widehat{f\,d\mu}(\eta)|^2\,d\eta \right)\,d\xi.
\end{aligned}
\end{equation*}
By Lemma \ref{lemma1}, it is less than or equal to
$$2^{j(d-s_\mu-\frac{d+2s_\mu-2}{4}+\epsilon)}\cdot 2^{j(d-2s_\mu)}||f||^2_{L^2(\mu)}= 2^{j(\frac{7d-14s_\mu+2}{4}+\epsilon)}||f||^2_{L^2(\mu)},$$
which is summable in $j$ if $s_\mu>\frac{d}{2}+\frac{1}{7}$, as desired.

\bibliographystyle{abbrv}
\bibliography{/Users/MacPro/Dropbox/Academic/paper/mybibtex.bib}

\begin{thebibliography}{10}

\bibitem{BIT16}
M.~Bennett, A.~Iosevich, and K.~Taylor.
\newblock Finite chains inside thin subsets of {$\Bbb{R}^d$}.
\newblock {\em Anal. PDE}, 9(3):597--614, 2016.

\bibitem{Erd05}
M.~B. Erdogan.
\newblock A bilinear {F}ourier extension theorem and applications to the
  distance set problem.
\newblock {\em Int. Math. Res. Not.}, (23):1411--1425, 2005.

\bibitem{EIT11}
S.~Eswarathasan, A.~Iosevich, and K.~Taylor.
\newblock Fourier integral operators, fractal sets, and the regular value
  theorem.
\newblock {\em Adv. Math.}, 228(4):2385--2402, 2011.

\bibitem{Fal85}
K.~J. Falconer.
\newblock On the {H}ausdorff dimensions of distance sets.
\newblock {\em Mathematika}, 32(2):206--212, 1985.

\bibitem{Fed69}
H.~Federer.
\newblock {\em Geometric measure theory}.
\newblock Die Grundlehren der mathematischen Wissenschaften, Band 153.
  Springer-Verlag New York Inc., New York, 1969.

\bibitem{GGIP15}
L.~Grafakos, A.~Greenleaf, A.~Iosevich, and E.~Palsson.
\newblock Multilinear generalized {R}adon transforms and point configurations.
\newblock {\em Forum Math.}, 27(4):2323--2360, 2015.

\bibitem{GI12}
A.~Greenleaf and A.~Iosevich.
\newblock On triangles determined by subsets of the {E}uclidean plane, the
  associated bilinear operators and applications to discrete geometry.
\newblock {\em Anal. PDE}, 5(2):397--409, 2012.

\bibitem{GILP15}
A.~Greenleaf, A.~Iosevich, B.~Liu, and E.~Palsson.
\newblock A group-theoretic viewpoint on {E}rd{\H o}s-{F}alconer problems and
  the {M}attila integral.
\newblock {\em Rev. Mat. Iberoam.}, 31(3):799--810, 2015.

\bibitem{GILP16}
A.~Greenleaf, A.~Iosevich, B.~Liu, and E.~Palsson.
\newblock An elementary approach to simplexes in thin subsets of euclidean
  space.
\newblock {\em http://arxiv.org/abs/1608.04777}, 2016.

\bibitem{GK15}
L.~Guth and N.~H. Katz.
\newblock On the {E}rd{\H o}s distinct distances problem in the plane.
\newblock {\em Ann. of Math. (2)}, 181(1):155--190, 2015.

\bibitem{IL17}
A.~Iosevich and B.~Liu.
\newblock Pinned distance problem, slicing measures and local smoothing
  estimates.
\newblock {\em arXiv preprint arXiv:1706.09851}, 2017.

\bibitem{IME16}
A.~Iosevich, M.~Mourgoglou, and E.~A. Palsson.
\newblock On angles determined by fractal subsets of the {E}uclidean space.
\newblock {\em Math. Res. Lett.}, 23(6):1737--1759, 2016.

\bibitem{ITU16}
A.~Iosevich, K.~Taylor, and I.~Uriarte-Tuero.
\newblock Pinned geometric configurations in euclidean space and riemannian
  manifolds.
\newblock {\em https://arxiv.org/pdf/1610.00349v1.pdf}, 2016.

\bibitem{Liu17}
B.~Liu.
\newblock Group actions, the mattila integral and continuous sum-product
  problems.
\newblock {\em arXiv preprint arXiv:1705.00560}, 2017.

\bibitem{Mat87}
P.~Mattila.
\newblock Spherical averages of {F}ourier transforms of measures with finite
  energy; dimension of intersections and distance sets.
\newblock {\em Mathematika}, 34(2):207--228, 1987.

\bibitem{Mat95}
P.~Mattila.
\newblock {\em Geometry of Sets and Measures in Euclidean Spaces: Fractals and
  Rectifiability}, volume~44 of {\em Cambridge Studies in Advanced
  Mathematics}.
\newblock Cambridge University Press, Cambridge, 1995.

\bibitem{PS00}
Y.~Peres and W.~Schlag.
\newblock Smoothness of projections, {B}ernoulli convolutions, and the
  dimension of exceptions.
\newblock {\em Duke Math. J.}, 102(2):193--251, 2000.

\bibitem{Wol99}
T.~Wolff.
\newblock Decay of circular means of {F}ourier transforms of measures.
\newblock {\em Internat. Math. Res. Notices}, (10):547--567, 1999.

\bibitem{Wol03}
T.~H. Wolff.
\newblock {\em Lectures on harmonic analysis}, volume~29 of {\em University
  Lecture Series}.
\newblock American Mathematical Society, Providence, RI, 2003.
\newblock With a foreword by Charles Fefferman and preface by Izabella \L aba,
  Edited by \L aba and Carol Shubin.

\end{thebibliography}

\end{document}